\documentclass[12pt,reqno]{article}
\usepackage[usenames]{color}
\usepackage{amssymb}
\usepackage{graphicx}
\usepackage{amscd}

\usepackage[colorlinks=true,
linkcolor=webgreen,
filecolor=webbrown,
citecolor=webgreen]{hyperref}

\definecolor{webgreen}{rgb}{0,.5,0}
\definecolor{webbrown}{rgb}{.6,0,0}

\usepackage{color}
\usepackage{fullpage}
\usepackage{float}

\usepackage{graphics,amsmath,amssymb}
\usepackage{amsthm}
\usepackage{amsfonts}
\usepackage{latexsym}
\usepackage{epsf}

\usepackage[mathscr]{eucal}
\usepackage[weather, misc]{ifsym}

.
.
.
.
.

\newcommand{\NN}{{\mathbb N}}
\newcommand{\ZZ}{{\mathbb Z}}
\newcommand{\CC}{{\mathbb C}}
\newcommand{\RR}{{\mathbb R}}
\newcommand{\Primes}{{\mathbb P}}
\newcommand{\SF}{{\mathbb S}}
\newcommand{\Mod}[1]{\, (\mathrm{mod}\, #1)}

\theoremstyle{plain}
\newtheorem{theorem}{Theorem}

\newtheorem{lemma}{Lemma}
\newtheorem{proposition}{Proposition}

\theoremstyle{definition}

\theoremstyle{remark}
\newtheorem{remark}{Remark}

\begin{document}
	
	\begin{center}
		\vskip 1cm{\LARGE\bf  On Ramanujan expansions and primes in arithmetic progressions
		}
		\vskip 1cm 
		\large
		Maurizio Laporta\\
		Universit\`a degli Studi di Napoli\\
		Dipartimento di Matematica e Applicazioni \\
		Complesso di Monte S. Angelo, Via Cinthia\\
		80126 Napoli (NA), Italy\\
		Orcid: 0000-0001-5091-8491\\
		\href{mailto:mlaporta@unina.it}{\tt mlaporta@unina.it}\\
	\end{center}

\begin{abstract} A celebrated theorem of Delange gives a sufficient condition for an arithmetic function to be the sum of the associated Ramanujan expansion with the coefficients provided by a previous result of Wintner.  By applying the Delange theorem to the correlation of the von Mangoldt function with its incomplete form,
	we deduce an inequality involving the counting function of the prime numbers in arithmetic progressions. A remarkable aspect is that such an inequality is equivalent to
	the famous conjectural formula
	by  Hardy and Littlewood for the twin primes.
\end{abstract}



\section{\large Introduction and statement of the main result}\label{sec1}

Some basic notations and definitions are included in the next  section. Further notations and definitions are introduced by their first occurrence. A celebrated theorem of Delange \cite{[De]}, \cite[Th. 3]{[M]}, \cite[VIII.2, Th. 2.1]{[ScSp]} states that if the arithmetic function $f:\NN\rightarrow\CC$ is such that the series
\begin{equation}\label{delange}
	\sum_{n=1}^{\infty}2^{\omega(n)}
	\frac{|(\mu\ast f)(n)|}{n}
\end{equation} 
converges,  then we have the absolutely convergent Ramanujan expansion
\begin{equation}\label{WRE}
	f(n)=\sum_{q=1}^{\infty}\widehat f(q)c_q(n),\quad \forall n\in\NN,
\end{equation} 
where $c_q(n)$ is the Ramanujan sum \cite{[R]} and the coefficients are \cite[\S 35]{[W]}
\begin{equation}\label{wincoeff}
	\widehat f(q)=\frac{1}{\varphi(q)}\lim_{x\to\infty}\frac{1}{x}\sum_{n\le x}f(n)c_q(n)=\frac{1}{q}\sum_{n=1}^{\infty}
	\frac{(\mu\ast f)(nq)}{n},\quad (q\in\NN).
\end{equation}
Here, $\omega(n)$ is
the number of the distinct prime factors of $n$, $\ast$ denotes
the Dirichlet product, $\mu$ is  the M\"obius function and $\varphi$ is the Euler totient. Let us also recall the von Mangoldt function \cite[\S1.4]{[IK]} 
$$
\Lambda(n)=-\sum_{d|n}\mu(d)\log d,
$$ 
which yields $\Lambda(n)=\log p$ if $n=p^\alpha$ for some prime number $p$ and $\alpha\in\NN$, $\Lambda(n)=0$ otherwise.  
By using the well-known property \cite[Th.\ 4.1]{[MV]}
\begin{equation}\label{RamSumConv}
	\sum_{q|m}c_q(n)=
	\begin{cases}
		m, &\text{if $m|n$,}
		\\
		0, &\text{otherwise,}
	\end{cases}
\end{equation} 
we see that for the incomplete $\Lambda$-function of range $N\in\NN$ \cite[\S19.2]{[IK]}
$$
n\in \NN\longrightarrow\ \Lambda_{_N}(n)=-\sum_{{d\le N}\atop {d|n}}\mu(d)\log d,
$$
the (finite) Ramanujan expansion 
\begin{equation}\label{lambdaFRE}
	\Lambda_{_N}(n)=-\sum_{d\le z}\frac{\mu(d)\log d}{d}\sum_{q|d}c_q(n)
	=\sum_{q\le N}\widehat{\Lambda_{_N}}(q)c_q(n), 
\end{equation}
holds with 
\begin{equation}\label{lambdabound}
	\widehat{\Lambda_{_N}}(q)= -\frac{\mu(q)}{q}\sum_{{d\le N/q}\atop {(d,q)=1}}\frac{\mu(d)\log(dq)}{d}\ll \frac{L^2}{q},\quad \forall q\le N.
\end{equation}
Hereafter, the symbol $(m,n)$ denotes the greatest common divisor of the integers $m,n$.  Moreover, we have set $L=\log N$ for brevity. 

Let us consider the correlation  of $\Lambda$ and $\Lambda_{_N}$, i.e., the arithmetic function
$$
n\in \NN\longrightarrow\ C_{\Lambda,\Lambda_{_N}}(N,n)= \sum_{m\le N}\Lambda(m)\Lambda{_N}(m+n).
$$
Further, for any  $N\in\NN$ and any $h\in\ZZ$ let us set
$$
\Delta(N,h)=\sum_{q\le N}\widehat{\Lambda_{_N}}(q)\sum_{k\in\ZZ_q^*}
\psi(N;q,k)\delta(h; q,k),
$$
where $\ZZ_q^*=\{m\in\NN\cap[1,q]: (m,q)=1\}$,
$$
\psi(N;q,k)=\sum_{n\le N\atop n\equiv k\, \Mod q}\Lambda(n)
$$
is the (weighted) counting function of prime numbers in arithmetic progressions, and
$$
\delta(h;q,k)=c_q(k+h)-\frac{\mu(q)}{\varphi(q)}c_q(h).
$$
This is the deviation of the Ramanujan sum $c_q(k+h)$, $k\in\ZZ_q^*$, from its arithmetic mean expressed by 
Cohen's identity as \cite[Cor.\ 7.2]{[Coh1]}
\begin{equation}\label{cohen}
	\frac{1}{\varphi(q)}\sum_{k\in\ZZ_q^*}
	c_q(k+h)=\frac{\mu(q)}{\varphi(q)}c_q(h).
\end{equation}

In Sect.\ \ref{sec4} we prove the following result.

\begin{theorem}\label{DelangeBV} Let the series
	\begin{equation}\label{DHCorr}
		\sum_{m=1}^\infty\frac{2^{\omega(m)}}{m}
		\Big|\sum_{d|m}\mu(d)C_{\Lambda,\Lambda_{_N}}(N,m/d)\Big|
	\end{equation}  
	be convergent for every sufficiently large $N\in\NN$. For any given $h\in\NN$ and  every real number $\varepsilon>0$ one has 
	\begin{equation}\label{th1}
		\Delta(N,h)\ll_{\varepsilon} (N+h)^\varepsilon.
	\end{equation} 
\end{theorem}

\begin{remark} Plainly, Cohen's identity \eqref{cohen}
	yields
	$$
	\Delta(N,h)=\sum_{q\le N}\widehat{\Lambda_{_N}}(q)\sum_{k\in\ZZ_q^*}
	\delta(h; q,k)\big(\psi(N;q,k)-M\big),
	$$
	provided that $M=M(N,q)$ does not depend on $k\in\ZZ_q^*$. In particular, \eqref{th1} is equivalent to
	$$
	\sum_{q\le N}\widehat{\Lambda_{_N}}(q)\sum_{k\in\ZZ_q^*}
	\delta(h; q,k)\Big(\psi(N;q,k)-\frac{1}{\varphi(q)}\sum_{n\le N\atop (n,q)=1}\Lambda(n)\Big)\ll_{\varepsilon} (N+h)^\varepsilon.
	$$
	In Sect.\ \ref{sec5} we give some further reformulations of $\Delta(N,h)$, where the function $\psi(N;q,k)$ is not explicitly involved.
\end{remark}

\begin{remark} Coppola \cite{[Co]} has proved that if $k\in\NN$ is such that $0<k<N^{1-\delta}$, with $\delta \in (0,1/2)$ fixed, then
	\begin{equation}\label{twinHL}
		\sum_{q\le N}\frac{\widehat{\Lambda_{_N}}(q) }{\varphi(q)}c_q(2k)\sum_{n\le N}\Lambda(n)c_{q}(n)=N\sum_{q=1}^{\infty}\frac{\mu^2(q)} {\varphi^2(q)}c_{q}(2k)+O
		\Big(\frac{N}{e^{c\sqrt{L}}}\Big), 
	\end{equation}
	where $c>0$ is an absolute constant. Here the main term on the right hand side is the one of the celebrated conjectural formula
	by  Hardy and Littlewood for the correlation
	$$
	h\in \NN\longrightarrow\ C_{\Lambda,\Lambda}(N,h)= \sum_{n\le N}\Lambda(n)\Lambda(n+h),
	$$
	i.e., for the (weighted) counting function of the number of $2k$-twin primes up to $N$ \cite[Conjecture B]{[HL]}. Now, writing
	$$
	C_{\Lambda,\Lambda}(N,h)=-\sum_{n\le N}\Lambda(n)\sum_{d|n+h}\mu(d)\log d, 
	$$
	observe that the conditions  $n\le N$ and $d|n+h$ yield $d\le N+h$  in the second sum,  so that $\Lambda(n+h)=\Lambda_{_{N+h}}(n+h)$.
	Consequently, 
	\begin{eqnarray}
		C_{\Lambda,\Lambda}(N,h)&=&\nonumber -\sum_{n\le N}\Lambda(n)\sum_{{d\le N}\atop {d|n+h}}\mu(d)\log d
		-\sum_{n\le N}\Lambda(n)\sum_{{N<d\le N+h}\atop {d|n+h}}\mu(d)\log d
		\\
		&=&\nonumber C_{\Lambda,\Lambda_{_N}}(N,h)-
		\sum_{N<d\le N+h}\mu(d)\log d\sum_{{n\le N}\atop {n\equiv -h\, (d)}}\Lambda(n) 
		\\
		&=&\nonumber C_{\Lambda,\Lambda_{_N}}(N,h)+O\big(hL \log(N+h)\big),
	\end{eqnarray}
	after noticing that for $d>N$ there is at most one $n\in\NN$ such that $n\le N$ and $n\equiv -h\Mod d$. In view of \eqref{twinHL} and Lemma \ref{CM} below (see Sect.\ \ref{sec3}), this reveals that the Delange hypothesis on (\ref{DHCorr}) implies the Hardy-Littlewood conjecture for the $2k$-twin primes. 
	This also emphasizes the strength of \eqref{th1} because,  after a quick inspection of \eqref{ShiftedCorrVM} in Sect.\ \ref{sec4}, it turns out  that such a far-reaching conjecture   is equivalent to $\Delta(N,2k)=o(N)$. Unfortunately, we do not know how to show a sufficient cancellation between the summands of $\Delta(N,2k)$ unconditionally. We hope that
	the different formulations of $\Delta(N,h)$ given in Sect.\ \ref{sec5} might provide with useful insights for 
	future considerations in this direction.
	Finally, we refer the reader to \cite{[GP]}, where for the first time it was indicated how heuristically the theory of Ramanujan expansions leads to the Hardy-Littlewood twin primes conjecture.
\end{remark}

\section{Notation}\label{sec2}
\begin{itemize}
	\item $\NN$ is the set of positive integer.
	\item $\Primes$ is the set of positive prime numbers; the letter $p$ (with or without subscript)  is reserved for the prime numbers.
	\item The prime power factorization of $n\in\NN$ is 
	$\prod_{p\in\Primes}p^{v_p(n)}$, where all but a finite number of the exponents $v_p(n)$ are zero.
	\item The symbol $(m,n)$ denotes the greatest common divisor of the integers $m,n$.
	\item $(h_1\ast h_2)(n)=\sum_{d|n}h_1(d)h_2(n/d)$ is
	the Dirichlet product of the arithmetic functions $h_1,h_2$.
	\item  $\omega(n)=\#\{p\in\Primes:\, p|n\},\ \forall n\in\NN$. 
	\item $\tau(n)=\#\{d\in\NN:\, d|n\},\ \forall n\in\NN$.
	\item $\mu$ is  the M\"obius function, i.e., $\mu(n)=(-1)^{\omega(n)}$ if $n$ is square-free, $\mu(n)=0$ otherwise.
	\item $\SF=\{q\in\NN:\, \mu(q)\not=0\}$ is the set of the square-free positive integers.
	\item Mainly within formul\ae, we often write $m\equiv\, n\ (k)$ to mean that $m\equiv\, n\, \Mod{k}$, i.e., $k$ divides $m-n$.
	\item Unless otherwise stated, in sums like $\sum_{d\le z}$ it is assumed that $d\in\NN$.
	\item The Ramanujan sum is
	$$
	c_q(n)=\sum_{j\in\ZZ_q^*}e\Big(\frac{jn}{q}\Big),
	$$
	where $\ZZ_q^*=\{m\in\NN\cap[1,q]: (m,q)=1\}$ and $e(x)=\exp(2\pi ix)$ for any real number $x$.
\end{itemize}

Warning. It is clear from \eqref{lambdabound} that if $q\not\in\SF$, then $\widehat{\Lambda_{_N}}(q)=0$. The reader is cautioned that most of the considerations in sections \ref{sec4} and \ref{sec5}  are valid only because it is often tacitly assumed that we are dealing with square-free integers of the $\widehat{\Lambda_{_N}}$ support.  For example, we shall freely use without explicit mention the fact that if $q\in\SF$, then $(d,q/d)=1$ for all $d|q$, so that $g(q)=g(d)g(q/d)$ for any multiplicative arithmetic function $g$ involved here.

\section{Lemmata}\label{sec3}

In the first lemma we resume some properties of the Ramanujan sums.

\begin{lemma}
	\begin{equation}\label{RamSum}
		c_q(n)=\varphi(q)\frac{\mu\big(q/(n,q)\big)} {\varphi\big(q/(n,q)\big)},
		\enspace 
		\forall q,n\in\NN;
	\end{equation}
	\begin{equation}\label{RamSumBound}
		|c_q(n)|\le (n,q),
		\enspace 
		\forall q,n\in\NN;
	\end{equation}
	\begin{equation}\label{ramaperiodic}
		\sum_{t=1}^qc_q(t)=
		\begin{cases}
			1, &\text{if $q=1$,}
			\\
			0, &\text{otherwise,}
		\end{cases}
		\quad \forall q\in\NN;
	\end{equation}
	\begin{equation}\label{ortoRam}
		\lim_{x\to \infty}\frac{1}{x}\sum_{h\le x}c_r(h+n)c_q(h)=
		\begin{cases}
			c_q(n)\ &\mbox{if $r=q$,}\\ 0\ &\mbox{otherwise,}\ 
		\end{cases} \quad \forall r,q\in\NN, \forall n\in\ZZ;
	\end{equation}
	\begin{equation}\label{toth}
		\sum_{k\in\ZZ_q^*}\chi(k)c_q(k+n)=d\,\mu(q/d)c_{q/d}(n)\chi^*(-n),\quad \forall q\in\SF, \forall n\in\ZZ,
	\end{equation}
	where $\chi$ is a Dirichlet character $\Mod q$ and $\chi^*$ is the primitive character $\Mod d$ that induces $\chi$, so that the conductor of $\chi$ is $d_\chi=d$.

\end{lemma}

\begin{proof} For \eqref{RamSum} and 
	\eqref{ortoRam}  we refer the reader respectively to \cite[Th.\ 4.1]{[MV]} and \cite[Th.\ 1]{[M]}. Note that \eqref{RamSumBound} is obviously true if $(n,q)=1$. In case $(n,q)>1$,  we apply \eqref{RamSum} and the well-known property \cite[Th.\ 2.4]{[A1]}
	$$
	\frac{\varphi(n)}{n}=\prod_{p|n}\big(1-\frac{1}{p}\big),\quad \forall n\in\NN\setminus\{1\},
	$$
	to see that
	\begin{eqnarray}
		|c_q(n)|&\le&\nonumber \frac{\varphi(q)} {\varphi\big(q/(n,q)\big)}=(n,q)\prod_{p|q}\big(1-\frac{1}{p}\big)
		\prod_{p|q/(n,q)}\big(1-\frac{1}{p}\big)^{-1}\\
		&=&\nonumber(n,q)\prod_{p|(n,q)}\big(1-\frac{1}{p}\big)
		\le (n,q).
	\end{eqnarray}
	In order to prove \eqref{ramaperiodic}, it suffices to use the definition of the Ramanujan sum and  the well-known property \cite[Th.\ 8.1]{[A1]}
	$$
	\sum_{t=1}^qe\Big(\frac{jt}{q}\Big)=\begin{cases}
		q, &\text{if $q|j$,}
		\\
		0, &\text{otherwise.}
	\end{cases}
	$$
	Finally, the last equation \eqref{toth} is a particular instance of
	Toth's generalization of Cohen's identity \cite[Cor.\ 2.6]{[T]}
	$$
	\sum_{k\in\ZZ_q^*}\chi(k)c_q(k+n)=d\varphi(q)
	\chi^*(-n)\sum_{t|q/d\atop (t,n)=1}\frac{t\mu\big(q/(td)\big)}{\varphi(td)},
	$$
	combined with the Brauer-Rademacher identity \cite[Cor.\ 34]{[Coh1]}, \cite[Ch.\ 2]{[MC]} to see that for $q\in\SF$ we can write
	$$
	\sum_{t|q/d\atop (t,n)=1}\frac{t\mu\big(q/(td)\big)}{\varphi(td)}=
	\frac{1}{\varphi(d)}\sum_{t|q/d\atop (t,n)=1}\frac{t\mu\big(q/(td)\big)}{\varphi(t)}=
	\frac{\mu(q/d)}{\varphi(q)}c_{q/d}(n).
	$$
	The lemma is completley proved.
\end{proof}
\begin{remark} Let us recall that a first well-known consequence  of \eqref{RamSum} is that
	$c_{q}(n)$ is multiplicative with respect to $q$.
\end{remark}

The next lemma is an application of Delange's aforementioned theorem to 
the correlation $C_{\Lambda,\Lambda_{_N}}(N,n)$. See the work of Coppola and Murty \cite{[CM]} for a more general account on the Ramanujan expansions of correlations. 

\begin{lemma}\label{CM} If the series \eqref{DHCorr} converges for some $N\in\NN$, then
	$$
	C_{\Lambda,\Lambda_{_N}}(N,h)=	\sum_{q\le N}\frac{\widehat{\Lambda_{_N}}(q) }{\varphi(q)}c_q(h)\sum_{n\le N}\Lambda(n)c_{q}(n),\ \forall h\in\NN.
	$$  
\end{lemma}

\begin{proof} Let us set $f(h)=C_{\Lambda,\Lambda_{_N}}(N,h)$ for simplicity, and observe that
	the claimed expansion follows from Delange's theorem once the coefficients $\widehat f(q)$ are determined as in \eqref{wincoeff}.
	Indeed, we use \eqref{lambdaFRE} to write
	$$
	\sum_{h\le x}f(h)c_q(h)=
	\sum_{r\le N}\widehat{\Lambda_{_N}}(r)\sum_{n\le N}\Lambda(n)\sum_{h\le x}c_r(n+h)c_q(h), 
	\enspace 
	\forall q\in \NN.
	$$
	Therefore,  by applying \eqref{ortoRam} we get
	\begin{eqnarray}
		\widehat f(q)&=&\nonumber \frac{1}{\varphi(q)}\sum_{r\le N}\widehat{\Lambda_{_N}}(r)\sum_{n\le N}\Lambda(n)
		\lim_{x\to \infty} \frac{1}{x}\sum_{h\le x}c_r(n+h)c_q(h)\\
		&=&\nonumber
		\begin{cases}
			\displaystyle\frac{\widehat{\Lambda_{_N}}(q)}{\varphi(q)}\sum_{n\le N}\Lambda(n)c_q(n) &\mbox{if $q\le N$,}\\ 0 &\mbox{otherwise}.\ 
		\end{cases}
	\end{eqnarray}
	The lemma is completely proved.
\end{proof}

\begin{remark} Wintner \cite[\S 34]{[W]} observed that, although the convergence 
	of 
	\begin{equation}\label{wa}
		\sum_{n=1}^{\infty}\frac{|(\mu\ast f)(n)|}{n}
	\end{equation} 
	implies the existence of the coefficients \eqref{wincoeff}, it
	is compatible with the divergence of the series
	\eqref{WRE}. On the other hand, it was pointed out by Cohen \cite{[Coh2]} that absolutely convergent expansions \eqref{WRE} can be deduced for some special classes of multiplicative functions by assuming just the convergence of \eqref{wa}, which is clearly a condition weaker than the Delange hypothesis on \eqref{delange}. Recently, Coppola has announced a result that would allow us to replace 
	the hypothesis of the convergence of the series \eqref{DHCorr} in the previous lemma by that of the convergence of  
	\eqref{wa} with $f(n)=C_{\Lambda,\Lambda_{_N}}(N,n)$ (see https://arxiv.org/pdf/2012.11231.pdf). Of course, with such a result in place, Theorem \ref{DelangeBV} would be refined accordingly. 
\end{remark}

\section{\large The  proof of Theorem  \ref{DelangeBV}}\label{sec4}

Bearing in mind the definitions of $\psi(N;q,k)$ and $\delta(h; q,k)$, let us use the expansion \eqref{lambdaFRE} and the fact
$c_{q}(n)=\mu(q)$ for $(n,q)=1$ to write
\begin{eqnarray}\label{ShiftedCorrVM}
	\Delta(N,h)&=&\sum_{q\le N}\widehat{\Lambda_{_N}}(q)\sum_{n\le N\atop (n,q)=1}\Lambda(n)\biggl(c_q(n+h)-\frac{c_{q}(n)c_q(h)}{\varphi(q)}\biggr)\\
	&=&\nonumber C_{\Lambda,\Lambda_{_N}}(N,h)-
	\sum_{q\le N}\frac{\widehat{\Lambda_{_N}}(q)}{\varphi(q)}c_q(h)\sum_{n\le N}\Lambda(n)c_{q}(n)
	-{\cal R}(N,h),
\end{eqnarray}
where we have set
$$
{\cal R}(N,h)=\sum_{n,q\le N\atop (n,q)>1}\Lambda(n)\widehat{\Lambda_{_N}}(q)\biggl(c_q(n+h)-\frac{c_{q}(n)c_q(h)}{\varphi(q)}\biggr).
$$
Thus, in view of Lemma \ref{CM} it suffices to show that
$$
{\cal R}(N,h)\ll_{\varepsilon} (N+h)^\varepsilon.
$$
For this purpose, we note first that, since $\Lambda(n)=0$, unless $n=p^\alpha$ with $p\in\Primes$, $\alpha\in\NN$,  
the condition $(q,n)>1$ for $q\in\SF$ in ${\cal R}(N,h)$  reduces to $(q,n)=(q,p^\alpha)=p$ with $p||q$, i.e., $p|q$ and $p^2\!{\not|q}$. Then, we can also assume that $q\in\SF\setminus\Primes$ because, taking 
$q=p$ and $n=p^\alpha$ in  ${\cal R}(N,h)$, immediately we see that
$$
c_p(p^\alpha+h)-\frac{c_{p}(p^\alpha)c_p(h)}{\varphi(p)}=c_p(h)-\frac{\varphi(p)c_p(h)}{\varphi(p)}=0.
$$
Hereafter, in sums denoted by the symbol $\displaystyle{\sum^*_q}$ we mean that $q\in\SF\setminus\Primes$.
Further,  we set 
$$
\SF_d=\{q\in\SF:\, (q,d)=1\}.
$$ 
Thus, using \eqref{RamSum},
let us write
\begin{eqnarray}
	{\cal R}(N,h)&=&\nonumber
	\sum_{p\le N}\log p\sum^*_{q\le N \atop  q\equiv 0\, (p) }\widehat{\Lambda_{_N}}(q)
	\sum_{\alpha\le L_pN}\biggl(c_q(p^\alpha+h)-\frac{c_{q}(p^\alpha)c_q(h)}{\varphi(q)}\biggr)
	\\
	&=&\nonumber
	\sum_{p\le N}\log p\sum_{\alpha\le L_pN}\sum_{t\in\SF\atop t|p^\alpha+h}\sum^*_{q\le N,\, q\equiv 0\, (p)\atop 
		(q,p^\alpha+h)=t }\widehat{\Lambda_{_N}}(q)\delta_h(p,\alpha,t,q),
\end{eqnarray}
where we have set $L_pN=[\log_pN]$ and
\begin{eqnarray}
	\delta_h(p,\alpha,t,q)&=&\nonumber
	\varphi(q)\frac{\mu(q/t)}{\varphi(q/t)}-\frac{c_{q/p}(p^\alpha)c_{p}(p^\alpha)c_q(h)}{\varphi(q)}
	\\
	&=&\nonumber
	\varphi(t)\mu(q/t)-\frac{\mu(q/p)c_q(h)}{\varphi(q/p)}.
\end{eqnarray}
Therefore, we have that
$$
{\cal R}(N,h)={\sum}_{I}-{\sum}_{II},
$$
with 
$$
{\sum}_{I}=\sum_{p\le N}\log p\sum_{\alpha\le L_pN}\sum_{t\in\SF\atop t|p^\alpha+h}\varphi(t)\sum_{q\in\SF_p,\, q\le N/p\atop 
	(pq,p^\alpha+h)=t }\widehat{\Lambda_{_N}}(pq)
\mu(pq/t),
$$
$$
{\sum}_{II}=\sum_{p\le N}\log p\sum_{\alpha\le L_pN}\sum_{t\in\SF\atop t|p^\alpha+h}\sum_{q\in\SF_p,\, q\le N/p\atop 
	(pq,p^\alpha+h)=t }{\widehat{\Lambda_{_N}}(pq)}
\frac{\mu(q)c_{pq}(h)}{\varphi(q)}.
$$
Now, let us recall that $v_p(m)$ denotes the nonnegative integer such that $p^{v_p(m)}||m$, for any given $m\in\NN$. Note that if $t\in\SF$, $v_p(t)=1$, and
$v_p(h)\not=0$, then the condition $(pq,p^\alpha+h)=t$, with $q\in\SF_p$, becomes  $(q,p^\alpha+h)=t/p$, with $v_p(t/p)=0$. On the other hand, if $v_p(h)=0$, then $(pq,p^\alpha+h)=t$ is equivalent to $(q,p^\alpha+h)=t$. Consequently, 
\begin{eqnarray}
	{\sum}_{I}&=&\nonumber
	-\sum_{p\le N\atop v_p(h)=0}\log p\sum_{\alpha\le L_pN}\sum_{t\in\SF\atop t|p^\alpha+h}\varphi(t)\sum_{q\in\SF_{p},\, q\le N/p\atop  
		(q,p^\alpha+h)=t }\widehat{\Lambda_{_N}}(pq)
	\mu(q/t)
	\\	
	&&\nonumber
	-\sum_{p\le N\atop v_p(h)\ge 1}\log p\sum_{\alpha\le L_pN}\sum_{{t\in\SF,\, v_p(t)=0}\atop t|p^\alpha+h}\varphi(t)\sum_{q\in\SF_p,\, q\le N/p\atop  
		(q,p^\alpha+h)=t }\widehat{\Lambda_{_N}}(pq)
	\mu(q/t)\\
	&&\nonumber +\sum_{p\le N\atop v_p(h)\ge 1}\varphi(p)\log p\sum_{\alpha\le L_pN}
	\sum_{{t\in\SF \atop v_p(t)=1}\atop t|p^\alpha+h}\varphi\Big(\frac{t}{p}\Big)\mu\Big(\frac{t}{p}\Big)
	\sum_{{q\in\SF_p,\atop q\le N/p}\atop  
		(q,p^\alpha+h)=\frac{t}{p}}\!\!\!\widehat{\Lambda_{_N}}(pq)
	\mu(q)\\
	&=&\nonumber -{\sum}_{I,1}-{\sum}_{I,2}+{\sum}_{I,3},
\end{eqnarray}
say. By applying \eqref{lambdabound} and the inequality for the divisor function,
$\tau(n)\ll_\varepsilon n^\varepsilon$  \cite[\S2.3]{[MV]}, we get
\begin{eqnarray}
	{\sum}_{I,1}
	&=&\nonumber\sum_{p\le N\atop v_p(h)=0}\log p\sum_{\alpha\le L_pN}\sum_{t\in\SF\atop t|p^\alpha+h}\varphi(t)\sum_{q\in\SF_{pt},\, q\le N/(pt)\atop  
		(q,p^\alpha+h)=1 }\widehat{\Lambda_{_N}}(pqt)
	\mu(q)\\
	&\ll&\nonumber L^2\sum_{p\le N}\frac{\log p}{p}\sum_{\alpha\le L_pN}\sum_{t|p^\alpha+h}\frac{\varphi(t)}{t}\sum_{q\le N}\frac{1}{q}\\
	&\ll&\nonumber
	L^3\sum_{p\le N}\frac{\log p}{p}\sum_{\alpha\le L_pN}\tau(p^\alpha+h)\\
	&\ll_\varepsilon&\nonumber (N+h)^\varepsilon L^4\sum_{p\le N}\frac{1}{p}\ll_\varepsilon (N+h)^\varepsilon L^4\log L.
\end{eqnarray}
We can proceed in a completely analogous way for the sums ${\sum}_{I,2}$ and ${\sum}_{I,3}$ (which make sense only for $h\ge 2$). Indeed, we see that
$$
{\sum}_{I,2}\ll_\varepsilon
(N+h)^\varepsilon L^4\sum_{p|h}\frac{1}{p}
\ll_\varepsilon (N+h)^\varepsilon L^4\log h.
$$
Further,
\begin{eqnarray}
	{\sum}_{I,3}
	&\nonumber=&\sum_{p\le N\atop v_p(h)\ge 1}\varphi(p)\log p\sum_{\alpha\le L_pN}
	\sum_{t\in\SF,\, v_p(t)=1\atop t|p^\alpha+h}\varphi(t/p)
	\sum_{q\in\SF_t,\, q\le N/t\atop  
		(q,p^\alpha+h)=1 }\widehat{\Lambda_{_N}}(qt)
	\mu(q)
	\\
	&\nonumber\ll& L^2\sum_{p\le N\atop v_p(h)\ge 1}\frac{\varphi(p)}{p}\log p\sum_{\alpha\le L_pN}
	\sum_{t\in\SF,\, v_p(t)=1\atop t|p^\alpha+h}\frac{\varphi(t/p)}{t/p}
	\sum_{q\le N}\frac{1}{q}
	\\
	&\ll&  \nonumber (N+h)^\varepsilon L^4\sum_{p|h}\frac{\varphi(p)}{p}\ll \omega(h)(N+h)^\varepsilon L^4.
\end{eqnarray}
Since $\omega(h)=\sum_{p|h}1\le \sum_{n|h}\Lambda(n)=\log h$ \cite[Th.\ 2.10]{[A1]}, we conclude that
$$
{\sum}_{I}\ll_\varepsilon 
(N+h)^\varepsilon L^4\log(hL).
$$
Now, let us turn our attention to ${\sum}_{II}$ and write
\begin{eqnarray}
	{\sum}_{II}&=&\nonumber
	-\sum_{p\le N\atop v_p(h)=0}\log p\sum_{\alpha\le L_pN}\sum_{t\in\SF\atop t|p^\alpha+h}\sum_{q\in\SF_{p},\, q\le N/p\atop  
		(q,p^\alpha+h)=t }\widehat{\Lambda_{_N}}(pq)
	\frac{\mu(q)c_{q}(h)}{\varphi(q)}
	\\	
	&&\nonumber
	+\sum_{p\le N\atop v_p(h)\ge 1}\varphi(p)\log p\sum_{\alpha\le L_pN}\sum_{{t\in\SF,\, v_p(t)=0}\atop t|p^\alpha+h}\sum_{q\in\SF_p,\, q\le N/p\atop  
		(q,p^\alpha+h)=t }\widehat{\Lambda_{_N}}(pq)
	\frac{\mu(q)c_{q}(h)}{\varphi(q)}\\
	&&\nonumber
	+\sum_{p\le N\atop v_p(h)\ge 1}\varphi(p)\log p\sum_{\alpha\le L_pN}
	\sum_{t\in\SF,\, v_p(t)=1\atop t|p^\alpha+h}
	\sum_{q\in\SF_p,\, q\le N/p\atop  
		(q,p^\alpha+h)=t/p}\widehat{\Lambda_{_N}}(pq)
	\frac{\mu(q)c_{q}(h)}{\varphi(q)}\\
	&=&\nonumber -{\sum}_{II,1}+{\sum}_{II,2}+{\sum}_{II,3},
\end{eqnarray}
say. Thus,  by using \eqref{RamSumBound}, we infer
\begin{eqnarray}
	{\sum}_{II,1}
	&\nonumber =&\sum_{p\le N\atop v_p(h)=0}\log p\sum_{\alpha\le L_pN}\sum_{t\in\SF\atop t|p^\alpha+h}\frac{\mu(t)}{\varphi(t)}\sum_{q\in\SF_{pt},\, q\le N/(pt)\atop  
		(q,p^\alpha+h)=1 }\widehat{\Lambda_{_N}}(pqt)
	\frac{\mu(q)c_{qt}(h)}{\varphi(q)}\\
	&\nonumber
	\ll& L^2\sum_{p\le N\atop v_p(h)=0}\frac{\log p}{p}\sum_{\alpha\le L_pN}\sum_{t\in\SF\atop t|p^\alpha+h}\frac{(t,h)}{t\varphi(t)}\sum_{q\le N}\frac{(q,h)}{q\varphi(q)}.
\end{eqnarray}
Now, let us apply the inequality  $\varphi(n)\gg n/\log\log n$ \cite[Th.\ 2.9]{[MV]} to see that
\begin{eqnarray}
	\sum_{q\le N}\frac{(q,h)}{q\varphi(q)}&\nonumber
	\ll&
	\log L
	\sum_{q\le N}\frac{(q,h)}{q^2}=
	\log L
	\sum_{d|h}d\sum_{q\le N\atop (q,h)=d}\frac{1}{q^2}
	\\
	&\nonumber =&
	\log L
	\sum_{d|h}\frac{1}{d}\sum_{m\le N/d\atop (dm,h)=d}\frac{1}{m^2}\ll \max\{1,\log h\}\log L.
\end{eqnarray}
Further, note that if $t\in\SF$, $t|p^\alpha+h$, and $(t,h)=d>1$, then necessarily $d=p$. Consequently, recalling that $v_p(h)=0$, one has
$$
\sum_{t\in\SF\atop t|p^\alpha+h}\frac{(t,h)}{t\varphi(t)}\ll
\log L
\sum_{t\in\SF\atop t|p^\alpha+h}\frac{1}{t^2}\ll \log L.
$$
Thus, 
$$
{\sum}_{II,1}
\ll \max\{1,\log h\}L^3\log^3 L.
$$
Analogously, for $h\ge 2$ we obtain
\begin{eqnarray}
	{\sum}_{II,2}
	&=&\nonumber
	\sum_{p\le N\atop v_p(h)\ge 1}\varphi(p)\log p\sum_{\alpha\le L_pN}\sum_{{t\in\SF\atop v_p(t)=0}\atop t|p^\alpha+h}\frac{\mu(t)}{\varphi(t)}\sum_{{q\in\SF_{pt}\atop q\le N/(pt)}\atop  
		(q,p^\alpha+h)=1}\!\!\!\widehat{\Lambda_{_N}}(pqt)
	\frac{\mu(q)c_{qt}\Big(\frac{h}{p}\Big)}{\varphi(q)}\\
	&\nonumber
	\ll& L^2\sum_{p\le N\atop v_p(h)\ge 1}\frac{\varphi(p)}{p}\log p\sum_{\alpha\le L_pN}\sum_{{t\in\SF,\, v_p(t)=0}\atop t|p^\alpha+h}\frac{(t,h/p)}{t\varphi(t)}\sum_{q\le N}
	\frac{(q,h/p)}{q\varphi(q)}\\
	&\nonumber
	\ll&  L^3\log^2L\max\{1,\log h\}\log h,
\end{eqnarray}
and
\begin{eqnarray}
	{\sum}_{II,3}
	&\nonumber
	\ll& L^2\sum_{p\le N\atop v_p(h)\ge 1}\varphi(p)\log p\sum_{\alpha\le L_pN}\sum_{{t\in\SF,\, v_p(t)=1}\atop t|p^\alpha+h}\frac{(h,t/p)}{t\varphi(t/p)}\sum_{q\le N}
	\frac{(q,h)}{q\varphi(q)}
	\\
	&\nonumber
	\ll& L^2\sum_{p\le N\atop v_p(h)\ge 1}\frac{\varphi(p)}{p}\log p\sum_{\alpha\le L_pN}\sum_{{t\in\SF,\, v_p(t)=1}\atop t|p^\alpha+h}\frac{1}{\varphi(t/p)t/p}\sum_{q\le N}
	\frac{(q,h)}{q\varphi(q)}\\
	&\nonumber
	\ll&  L^3\log^2L\max\{1,\log h\}\log h.
\end{eqnarray}
Therefore, since $\varepsilon>0$ is arbitrarily small, we can conclude that
$$
{\cal R}(N,h)={\sum}_{I}-{\sum}_{II}\ll_{\varepsilon} 
(N+h)^\varepsilon.
$$
Theorem \ref{DelangeBV} is completely proved.

\section{\large Some reformulations of $\Delta(N,h)$}\label{sec5}

Let us set 
$$
S_N(q,k;r)=\sum_{n\le N\atop n\equiv k\, (q)}c_r(n), 
$$ 
so that 
$$
\sum_{n\le N\atop (n,q)=1}c_r(n)=\sum_{k\in\ZZ_q^*}S_N(q,k;r).
$$
We postpone some properties of $S_N(q,k;r)$ until the next section, because they are not used in what follows here, but still they might be interesting in their own right. Here we use such sums to provide with alternative expressions for $\Delta(N,h)$ "without primes". 
\begin{theorem} For any given $h,n\in\NN$, one has:
	\smallskip
	
	\noindent
	1)	$\displaystyle\Delta(N,h)=\sum_{3\le q\le N}\widehat{\Lambda_{_N}}(q)\sum_{r\le N\atop (r,q)=1}\widehat{\Lambda_{_N}}(r)D_N(h; q,r)$, where
	$$
	D_N(h; q,r)=
	\sum_{n\le N\atop (n,q)=1}c_r(n)\delta(h;q,n)=\sum_{k\in\ZZ_q^*}S_N(q,k;r)\delta(h;q,k).
	$$
	2) $\displaystyle\Delta(N,h)=\sum_{t|h}\mu(t)\varphi(t)
	\sum_{3\le q\le N\atop (q,h)=t}\frac{\widehat{\Lambda_{_N}}(q)}{\varphi(q)}
	\sum_{r\le N\atop (r,q)=1}\widehat{\Lambda_{_N}}(r) 
	D_N(t,h; q,r)$, where
	\begin{displaymath}
		D_N(t,h; q,r)=
		\begin{cases}
			\displaystyle\mu(q)\varphi(q)\sum_{m|q} m\varphi(m)\sum_{{n\le N\atop (n,q)=1}\atop n+h\equiv 0\, (m)}c_r(n), &\text{if $t=1$,}\\
			\displaystyle\sum_{m|q\atop (m,h)=1} m\varphi(m)\sum_{{n\le N\atop (n,q)=1}\atop n+h\equiv 0\, (m)}c_r(n), &\text{otherwise.}
		\end{cases}
	\end{displaymath}
\end{theorem}

\begin{proof} 1) First, by observing  \eqref{ShiftedCorrVM}, it is plain that $q=1$ gives null contribution to $\Delta(N,h)$. So does $q=2$ because for any odd $n$ one has
	$$
	\delta(h;2,n)=c_2(n+h)-\frac{\mu(2)c_2(h)}{\varphi(2)}=c_2(1+h)+c_2(h)=0,\ \forall h\in\ZZ.
	$$
	Then, after recalling \eqref{lambdabound}, note that, for any 
	$n\in\NN$ such that $n\le N$, $(n,q)=1$, one has
	\begin{eqnarray}
		\Lambda(n)=\Lambda_{_N}(n)&=&\nonumber-\sum_{d\le N\atop d|n, (d,q)=1}\mu(d)\log d
		=-\sum_{d\le N\atop (d,q)=1}\frac{\mu(d)\log d}{d}\sum_{r|d}c_r(n)\\
		&=&\nonumber\sum_{r\le N\atop (r,q)=1}\widehat{\Lambda_{_N}}(r)c_r(n).
	\end{eqnarray}
	Thus, the first equality for $\Delta(N,h)$ follows by inserting the latter expression inside the reformulation in  \eqref{ShiftedCorrVM}. 
	
	2) By using  the second expression of $D_N(h; q,r)$ above, Cohen's identity \eqref{cohen}, 
	the orthogonality of Dirichlet charachters $\chi\, \Mod q$ \cite[\S 3.3]{[IK]}, and \eqref{toth},  we get 
	\begin{eqnarray}
		D_N(h; q,r)
		&=\nonumber&	\sum_{k\in\ZZ_q^*}c_q(k+h)S_N(q,k;r)
		-\frac{\mu(q)c_q(h)}{\varphi(q)}\sum_{n\le N\atop (n,q)=1}c_r(n)\\
		&=\nonumber&
		\sum_{k\in\ZZ_q^*}
		c_q(k+h)\biggl(\sum_{n\le N\atop n\equiv k\, (q)}c_r(n)-\frac{1}{\varphi(q)}\sum_{n\le N\atop (n,q)=1}c_r(n)
		\biggr)\\
		&=\nonumber&
		\frac{1}{\varphi(q)}\sum_{\chi\not=\chi_0\Mod q}\Upsilon_r(N,\overline{\chi})
		\sum_{k\in\ZZ_q^*}\chi(k)
		c_q(k+h)
		\\
		&=\nonumber&\frac{1}{\varphi(q)}\sum_{\chi\not=\chi_0\Mod q\atop d_\chi=d}
		\mu(q/d)c_{q/d}(h)\Upsilon_r(N,\overline{\chi})\chi^*(-h)\, d,
	\end{eqnarray}
	where $\chi_0$ is the principal charachter and we have set 
	$$
	\Upsilon_r(N,\overline{\chi})=\sum_{n\le N}\overline{\chi}(n)c_r(n). 
	$$
	Now, given any $q\in\SF$ and any $d|q$ such that $(d,h)=1$, note that if $(q,h)=t$, then 
	$(q/d,h)=t$. In particular, from  \eqref{RamSum} it follows that
	$\mu(q/d)c_{q/d}(h)=\mu(t)\varphi(t)$.
	Thus, being plain that $\chi^*(-h)=0$ if $(d_\chi,h)>1$,  we can write
	\begin{eqnarray}
		\Delta(N,h)
		&=\nonumber&
		\sum_{3\le q\le N}\frac{\widehat{\Lambda_{_N}}(q)}{\varphi(q)}
		\sum_{r\le N\atop (r,q)=1}\widehat{\Lambda_{_N}}(r) 
		\sum_{{\chi\not=\chi_0\atop\Mod q}\atop d_\chi=d}\mu\Big(\frac{q}{d}\Big)c_{\frac{q}{d}}(h)\, \Upsilon_r(N,\overline{\chi})\,\chi^*(-h)\, d\\
		&=\nonumber&
		\sum_{t|h}\mu(t)\varphi(t)
		\sum_{3\le q\le N\atop (q,h)=t}\frac{\widehat{\Lambda_{_N}}(q)}{\varphi(q)}
		\sum_{r\le N\atop (r,q)=1}\widehat{\Lambda_{_N}}(r)\Phi_{N,h}(q,r),
	\end{eqnarray}
	where
	\begin{eqnarray}	
		\Phi_{N,h}(q,r)&=\nonumber&\sum_{\chi\not=\chi_0\Mod q\atop d=d_\chi}\Upsilon_r(N,\overline{\chi})\,\chi^*(-h)\, d\\
		&=\nonumber&
		\sum_{n\le N\atop (n,q)=1}c_r(n)\sum_{\chi\not=\chi_0\Mod {q}} \overline{\chi}(n)\,\chi^*(-h)\, d_\chi.
	\end{eqnarray}
	For $(n,q)=1$, let $n'\in\NN$ be such that $n'n\equiv 1\, \Mod q$.
	Since there is a canonical bijection between the set of Dirichlet characters modulo $q$ and the set of primitive Dirichlet characters whose conductor divides $q$, we have
	\begin{eqnarray}
		\sum_{\chi\not=\chi_0\Mod {q}} \overline{\chi}(n)\,\chi^*(-h)\, d_\chi
		&=\nonumber&
		\sum_{d|q\atop (d,h)=1}d\sum_{\chi^* \Mod {d}\atop {\rm primitive}} \overline{\chi}^*(n)\,\chi^*(-h)
		\\
		&=\nonumber&
		\sum_{d|q\atop (d,h)=1}d\sum_{\chi^* \Mod {d}\atop {\rm primitive}} \chi^*(-hn')
		\\
		&=\nonumber&
		\sum_{d|q\atop (d,h)=1}d\sum_{m|(d,-hn'-1)} \varphi(m)\mu(d/m)
		\\
		&=\nonumber&
		\sum_{d|q\atop (d,h)=1}d\sum_{m|(d,n+h)} \varphi(m)\mu(d/m),
	\end{eqnarray}
	where we have used the formula \cite[\S 3.3]{[IK]}
	$$
	\sum_{\chi^* \Mod {d}\atop {\rm primitive}} \chi^*(a)=\sum_{m|(d,a-1)} \varphi(m)\mu(d/m),\ \hbox{if}\ (a,d)=1.
	$$
	Therefore,
	\begin{eqnarray}
		\Phi_{N,h}(q,r)
		&=\nonumber&
		\sum_{n\le N\atop (n,q)=1}c_r(n)\sum_{d|q\atop (d,h)=1}d\sum_{m|(d,n+h)} \varphi(m)\mu(d/m)\\
		&=\nonumber&
		\sum_{d|q\atop (d,h)=1}d\sum_{m|d} \varphi(m)\mu(d/m)
		T_{N,h}(m;q,r)
		\\
		&=\nonumber&
		\sum_{m|q\atop (m,h)=1}\varphi(m)T_{N,h}(m;q,r)\sum_{d|q, (d,h)=1\atop d\equiv 0\, (m)}d\mu(d/m),
	\end{eqnarray}
	where we have set 
	$$
	T_{N,h}(m;q,r)=\sum_{{n\le N\atop (n,q)=1}\atop n+h\equiv 0\, (m)}c_r(n).
	$$
	Hence, the conclusion follows after observing that, for any divisor $m$ of $q$ such that
	$(m,h)=1$, one has
	\begin{eqnarray}
		\sum_{d|q, (d,h)=1\atop d\equiv 0\, (m)}d\mu(d/m)&=\nonumber&
		m\sum_{d|q\atop (d,h)=1}d\mu(d)=m\sum_{d|q}d\mu(d)\sum_{s|(d,h)}\mu(s)
		\\
		&=\nonumber&
		m\sum_{s|(q,h)}\mu(s)\sum_{d|q\atop d\equiv 0\, (s)}d\mu(d)=
		m\sum_{s|(q,h)}\mu(s)^2s\sum_{d|q/s}d\mu(d)
		\\
		&=\nonumber&
		m\sum_{s|(q,h)}s\mu(q/s)\varphi(q/s)
		\\
		&=\nonumber&\begin{cases}
			\displaystyle \mu(q)\varphi(q)m, &\text{if $(q,h)=1$,}\\
			m, &\text{otherwise.}
		\end{cases}
	\end{eqnarray}
	The theorem is completely proved.
\end{proof}

\section{\large Appendix: some properties of $S_N(q,k;r)$}\label{sec6}

Here we prove some properties of the sum
$S_N(q,k;r)$ defined in the previous section. 
Such properties are either new or appeared to be missing from the literature. The first
one is a quantitative version of a known property \cite[VIII.8, Ex. 2]{[ScSp]}. Before going to the next propositions, let us recall that $[\beta]$ and $\Vert \beta\Vert$ denote respectively  the integer part of $\beta\in\RR$ and 
the distance of $\beta$ from the nearest integer.

\begin{proposition} Let $q,r,N\in\NN$ and $k\in\ZZ$ be given.
	
	\noindent	
	1) 
	$$
	S_N(q,k;r)=
	\begin{cases}
		\displaystyle	\frac{N-k}{q}c_r(k)+O\big(|c_r(k)|\big), &\text{if $r|q$,}\\
		O(r\log r),  &\text{otherwise.}
	\end{cases}
	$$
	2) Let us assume that $(q,r)=1$ and let $q',r'\in\NN$ such that $rr'\equiv 1\, (q)$ and $qq'\equiv 1\, (r)$. Further, let $Q,R$ be the non-negative integers such that $N=Qqr+R$ with $0\le R<qr$, i.e., $R=N-[N/qr]qr$. Then
	\begin{displaymath}
		S_N(q,k;r)=\sum_{t=1}^rc_r(t)\Big([N/qr]+\nu_N(t,k)\Big)=
		\begin{cases}
			[N/q]+\nu_N(1,k), &\text{if $r=1$,}\\
			\displaystyle{\sum_{t=1}^rc_r(t)\nu_N(t,k)}, &\text{otherwise,}
		\end{cases}
	\end{displaymath}
	where
	\begin{displaymath}
		\nu_N(t,k)=
		\begin{cases}
			1, &\text{if $kr'r+tq'q\equiv n\ (qr)$ for some integer $n\in[1,R]$,}\\
			0, &\text{otherwise.}
		\end{cases}
	\end{displaymath}
\end{proposition}

\begin{proof} 1)  If $r|q$, then
	\begin{eqnarray}
		S_N(q,k;r)&=&\nonumber\sum_{t=1}^rc_r(t)\sum_{{n\le N\atop n\equiv k\, (q)}\atop n\equiv t\, (r)}1=
		c_r(k)\sum_{n\le N\atop n\equiv k\, (q)}1\\
		&=&\nonumber
		c_r(k)\biggl[\frac{N-k}{q}\biggr]=
		\frac{N-k}{q}c_r(k)+O\big(|c_r(k)|\big).
	\end{eqnarray}	
	Now, we consider the case $r\not|q$ and write
	\begin{eqnarray}
		S_N(q,k;r)&=&\nonumber\sum_{m\le (N-k)/q}c_r(k+qm)\\
		&=&\nonumber
		\sum_{j\in\ZZ_r^*}e(jk/r)\sum_{m\le (N-k)/q}e(jqm/r)
		\ll \sum_{j\in\ZZ_r^*}\biggl\Vert \frac{jq}{r}\biggr\Vert^{-1},
	\end{eqnarray}
	where we have applied the inequality \cite[Ch.\ 8]{[IK]}
	$$
	\sum_{m\le x}e(m\beta)\ll \min(x, \Vert \beta\Vert^{-1})\quad \forall x,\beta\in\RR.
	$$
	Let us set $t=(q,r)$. Since $(q/t,r/t)=1$, it turns out that 
	$$
	\sum_{j\in\ZZ_r^*}\biggl\Vert \frac{jq}{r}\biggr\Vert^{-1}
	=
	\sum_{{j\le r}\atop {j\not \equiv 0\, (r/t)} }\biggl\Vert \frac{jq/t}{r/t}\biggr\Vert^{-1}
	\le t \sum_{j'<r/t}\biggl\Vert \frac{j'}{r/t}\biggr\Vert^{-1}
	\ll r \sum_{j'<r/t}{1\over {j'}}
	\ll r\log r.
	$$
	2)  By applying the Chinese Remainder Theorem \cite[Th.\ 5.26]{[A1]}, we see that
	$$
	S_N(q,k;r)=
	\sum_{t=1}^rc_r(t)\sum_{{n\le N\atop n\equiv k\, (q)}\atop n\equiv t\, (r)}1=
	\sum_{t=1}^rc_r(t)\sum_{n\le N\atop n\equiv kr'r+tq'q\, (qr)}1.
	$$
	The conclusion follows immediately by using \eqref{ramaperiodic}, after noticing that 
	$$
	\sum_{n\le N\atop n\equiv kr'r+tq'q\, (qr)}1=[N/qr]+\nu_N(t,k).
	$$
	The proposition is completely proved.
\end{proof}

\begin{proposition} Let $q,r,N\in\NN$ be coprime, with $r\ge 2$, and let $k\in\ZZ$ be given.
	
	\noindent	
	1) $S_{N_1}(q,k;r)=S_{N_2}(q,k;r)$  for any $N_1,N_2\in\NN$ such that $N_1\equiv N_2\ \Mod {qr}$.
	
	\noindent	
	2) If, in addition, $q\ge 2$ and $k\in\ZZ_q^*$, then
	$$
	\sum_{N=1}^{qr}S_N(q,k;r)=-\sum_{N=1}^{qr}S_N(q,q-k;r).
	$$
\end{proposition}

\begin{proof} 1) It suffices to observe that $S_N(q,k;r)=0$ for any $N\equiv 0\ \Mod {qr}$ by applying the second property of the previous proposition.
	
	\noindent
	2) As before, we write
	$$
	\sum_{N=1}^{qr}S_N(q,k;r)=\sum_{t=1}^rc_r(t)\sum_{N=1}^{qr}\sum_{n\le N\atop n\equiv kr'r+tq'q\, (qr)}1.
	$$
	Now, we see that
	\begin{eqnarray}
		\sum_{N=1}^{qr}\sum_{n\le N\atop n\equiv kr'r+tq'q\, (qr)}1
		&=&\nonumber\sum_{n\le qr\atop n\equiv kr'r+tq'q\, (qr)}\sum_{N=n}^{qr}1\\
		&=&\nonumber
		\sum_{n\le qr\atop n\equiv kr'r+tq'q\, (qr)}(qr-n+1)=qr+1-n(k,t),
	\end{eqnarray}
	where  $n(k,t)$ denotes the only solution of $x\equiv kr'r+tq'q\, \Mod {qr}$ such that
	$1\le x\le qr$.
	Thus, by using \eqref{ramaperiodic} one has
	$$
	\sum_{N=1}^{qr}S_N(q,k;r)=-\sum_{t=1}^rc_r(t)n(k,t),
	$$
	$$
	\sum_{N=1}^{qr}S_N(q,q-k;r)=-\sum_{t=1}^rc_r(t)n(q-k,t).
	$$
	Since $c_r(r-t)=c_r(-t)=c_r(t)$ for any $t\in\{1,\ldots,r\}$, we see that
	$$
	\sum_{N=1}^{qr}S_N(q,q-k;r)
	=-\sum_{t=1}^rc_r(r-t)n(q-k,t)
	=-\sum_{s=0}^{r-1}c_r(s)n(q-k,r-s).
	$$
	The conclusion follows after noticing that $c_r(0)=c_r(r)$ and for any $s\in\{0,1,\ldots,r-1\}$ one has
	
	$$
	n(q-k,r-s)\equiv (q-k)r'r+(r-s)q'q \equiv -n(k,s)\, \Mod {qr}.
	$$
	The proposition is completely proved.
\end{proof}
\begin{remark} An immediate consequence of the latter property is that 
	$$
	\sum_{N=1}^{qr}\sum_{k\in\ZZ_q^*}S_N(q,k;r)=\sum_{N=1}^{qr}\sum_{n\le N\atop (n,q)=1}c_r(n)=0
	$$
	for any coprime $q,r\ge 2$.
\end{remark}

\end{document}